\documentclass[a4paper,12pt]{amsart}

\usepackage{a4wide}
\usepackage{pgf,tikz}
\usepackage{amssymb}
\usepackage{enumerate}

\usepackage{amsmath}

\newif\ifdetails
\detailstrue
\newcommand{\DETAIL}[1]%
{\ifdetails\par\fbox{\begin{minipage}{0.9\linewidth}\textit{Detail:}
      #1\end{minipage}}\par\fi}
\newcommand{\TODO}[1]%
{\ifdetails\par\fbox{\begin{minipage}{0.9\linewidth}\textbf{TODO:}
      #1\end{minipage}}\par\fi}

\usepackage{makecell}
\usepackage{relsize}

\newtheorem{lemma}{Lemma}
\newtheorem{proposition}[lemma]{Proposition}
\newtheorem{theorem}[lemma]{Theorem}

\theoremstyle{remark}

\usepackage{caption}
\usepackage{subcaption}
\usepackage{float} 
\usepackage[pdftex,a4paper,
citecolor = blue, colorlinks=true,urlcolor=blue]{hyperref}
\usepackage{mathrsfs}
\usetikzlibrary{arrows}

\newcommand{\old}[1]{{}}

\title[\tiny{Graphs/unicyclic graphs with extremal number of connected induced subgraphs}]{Graphs and unicyclic graphs with extremal number of connected induced subgraphs}

\author{Audace A. V. Dossou-Olory}

\thanks{The author was supported by the University of Johannesburg. This work started while the author was still affiliated to Stellenbosch University}

\address{Audace A. V. Dossou-Olory \\ Department of Pure and Applied Mathematics \\ University of Johannesburg \\ P.O. Box 524, Auckland Park, Johannesburg 2006\\ South Africa}
\email{audace@aims.ac.za}

\subjclass[2010]{Primary 05C30; secondary 05C35, 05C05}
\keywords{induced subgraphs, connected graphs, unicyclic graphs, tadpole graphs, extremal graph structures, path, complete graph}

\begin{document}

\begin{abstract}
Over all graphs (or unicyclic graphs) of a given order, we characterise those graphs that minimise or maximise the number of connected induced subgraphs. For each of these classes, we find that the graphs that minimise the number of connected induced subgraphs coincide with those that are known to maximise the Wiener index (the sum of the distances between all unordered pairs of vertices), and vice versa. For every $k$, we also determine the connected graphs that are extremal with respect to the number of $k$-vertex connected induced subgraphs. We show that, in contrast to the minimum which is uniquely realised by the path, the maximum value is attained by a rich class of connected graphs.
\end{abstract}

\maketitle

\section{Introduction and selected previous results}

Counting and understanding graph structures with particular properties has many applications, especially to network theory, computer science, biology and chemistry. For instance, graphs can represent biological networks at the molecular or species level (protein interactions, gene regulation, etc.)~\cite{lemons2011hierarchical}. The topological structure of an interconnection network is a \textit{connected} graph where, for example, vertices are processors and edges represent links between them~\cite{chen2009induced}. In chemical networks, vertices are atoms and edges represent their bonds. An important question is to find all matches of a specific motif within a larger network (the \textit{subgraph} isomorphism problem~\cite{cordella2004sub}, or the \textit{induced subgraph} isomorphism problem~\cite{chen2008understanding}). Both cases are known to be in general NP-complete~\cite{garey1979guide}, although in some instances (such as planar graphs), efficient algorithms are available~\cite{eppstein2002subgraph}. A step to these problems usually consists of enumerating all possible subgraphs or induced subgraphs of the network~\cite{wernicke2006fanmod}.

This paper discusses the number of connected induced subgraphs of a finite and simple (no loops, no parallel edges, undirected) graph with a particular emphasis on connected induced subgraphs of a connected graph. Let $G$ be a simple graph consisting of a finite (but not empty) set $V(G)$ of vertices and a finite set $E(G)$ of edges. A graph $H$ such that $\emptyset \neq V(H) \subseteq V(G)$ and $E(H) \subseteq E(G)$ is called a subgraph of $G$ (we do not consider the empty graph!). A graph formed from $G$ by taking a nonempty subset $S$ of vertices of $G$ and all edges incident with the vertices in $S$ is called an induced subgraph of $G$. In this paper, we are concerned with the problem of determining the minimum and maximum number of subgraphs or connected induced subgraphs of $G$, and also characterising the extremal graphs. This problem will be considered for certain types of graphs all sharing the same number of vertices.

\medskip
A graph with no edge is called an \textit{edgeless} graph, while a graph with an edge between every pair of distinct vertices is called a \textit{complete} graph. It is trivial that among all graphs having $n$ vertices, precisely the edgeless graph $E_n$ has the minimum number of subgraphs, and the maximum number of subgraphs is uniquely attained by the complete graph $K_n$. Distinct vertices $u,v \in V(G)$ are said to be connected in $G$ if there is a path from $u$ to $v$ in $G$. The graph $G$ is connected if and only if any two distinct vertices of $G$ are connected in $G$. Note that the notions of subgraph and induced subgraph coincide for the edgeless graph only. On the other hand, adding an edge $e$ between two nonadjacent vertices of a graph increases its number of connected subgraphs by at least one (namely, the graph induced by the endvertices of $e$ itself). Thus, the complete graph remains the only graph of order $n$ having the maximum number of connected subgraphs. The problem becomes more interesting when one considers the number of connected subgraphs of a connected graph. 

\medskip
Tittmann et al.~\cite{tittmann2011enumeration} enumerated the number of connected components in induced subgraphs by means of a generating function approach. Yan and Yeh~\cite{yan2006enumeration} gave a linear-time algorithm for counting the sum of weights of subtrees of a tree (connected acyclic graph). In~\cite{yan2006enumeration}, Yan and Yeh also asked for methods to enumerate connected subgraphs of a connected graph. Very recently, Kroeker et al.~\cite{kroeker2017meane} investigated the extremal structures for the mean order of connected induced subgraphs among the so-called cographs (graphs containing no induced path of order $4$). Our main interest in this paper is to know the minimum and maximum number of connected induced subgraphs that a connected graph of a given order can contain; the approach we use here does not involve generating functions. The enumeration of connected induced subgraphs of a fixed order $k$ is also important in many applications~\cite{kashani2009kavosh}, as eg. when solving certain fixed-cardinality optimisation problems~\cite{komusiewicz2015algorithmic}. Very recently, Komusiewicz and Sommer~\cite{komusiewicz2019enumerating} studied algorithms for the problem of enumerating all $k$-vertex connected induced subgraphs of a graph.

We remark that the notions of connected subgraph and connected induced subgraph coincide for trees only. This follows from the simple fact that every connected graph of order $n$ has at least $n-1$ edges with equality if and only if the graph is a tree. It is known (see Sz{\'e}kely-Wang~\cite{szekely2005subtrees}) that among all trees of order $n$, the path (resp. star) minimises (resp. maximises) the number of subtrees. 
	
\begin{theorem}[\cite{szekely2005subtrees}]\label{Szek:Theo}
The $n$-vertex path has $\binom{n+1}{2}$ subtrees, fewer than any other tree of order $n$. The $n$-vertex star has $n-1+2^{n-1}$ subtrees, more than any other tree of order $n$.
\end{theorem}		

For every $2<k<n $, it is also known (see Lemma~5.1 in~Jamison~\cite{Jamison1983} or Lemma~5.2.2 in~\cite{Mol2016}) that paths uniquely minimise the number of $k$-vertex subtrees among all $n$-vertex trees.

\begin{lemma}[\cite{Jamison1983}]\label{Jam:Lem}
For every $2<k<n $, the $n$-vertex path has fewer subtrees of order $k$ than any other tree of order $n$. 
\end{lemma}

Section~\ref{Sec:NkG} essentially contains extremal results on the number of $k$-vertex connected induced subgraphs of a connected graph of order $n$. It is shown that the path (resp. complete graph) provides the minimum (resp. maximum) number of $k$-vertex connected induced subgraphs among all connected graphs of order $n$. Moreover, we show that the maximum value is attained by all $n$-vertex graphs that result from removing exactly $l \leq n/2$ independent edges (edges sharing no common vertex) in the complete graph $K_n~(n\geq 3)$. Section~\ref{Sec:Uni} is concerned with a particularly well studied class of tree-like structure as a sole subject. We consider (connected) unicyclic graphs of a given order and investigate which unicyclic graphs minimise or maximise the number of connected induced subgraphs.

\medskip
A \textit{unicyclic} graph is a connected graph which contains exactly one cycle. The number of unicyclic graphs of a fixed order $n>2$ begins $$1, 2, 5, 13, 33, 89, 240, 657, 1806, 5026, 13999, 39260, 110381, 311465, 880840, \ldots$$ For more information on this sequence, see the On-Line Encyclopedia of Integer Sequences~\cite{oeis2018} under~\url{A001429}. Clearly, the complete graph is no longer extremal among unicyclic graphs of order $n> 3$ ($K_3$ is the only unicyclic graph of order $3$). By $G_{3,n-3}$, we mean the connected graph obtained by identifying a vertex of $K_3$ with a leaf of the path of order $n-2$. By $Q_n$, we mean the graph obtained by joining two leaves of the $n$-vertex star by an edge. It will be shown that $G_{3,n-3}$ is the unicylic graph of order $n$ with the smallest number of connected induced subgraphs (Theorem~\ref{MinforUnic}), while $Q_n$ is the unicylic graph of order $n$ with the greatest number of connected induced subgraphs (Theorem~\ref{MaxUnic}). 

It occurs very often that a certain tree is extremal within a given class of trees with respect to several graph invariants (the number of subtrees and the Wiener index, for instance). This also holds in our current context: the unicyclic graphs that are found to be extremal (see Section~\ref{Sec:Uni}) for the number of connected induced subgraphs were previously shown to be extremal for the Wiener index and the energy (among others). A number of other graph invariants were studied in various subclasses of unicyclic graphs. This includes the energy (sum of the absolute values of the eigenvalues) and two closely related parameters, namely the Merrifield-Simmons index (number of independent sets) and the Hosoya index (number of matchings). Hou~\cite{hou2001unicyclic} determined the unicyclic graphs of order $n$ with the minimal energy. Li and Zhou~\cite{li2008minimal} found the graphs with the minimal energy among all unicyclic graphs of order $n$ and diameter $d$.  Hou, Gutman and Woo~\cite{hou2002unicyclic} characterised the unicyclic bipartite graphs (that are not cycle) of order $n$ having the maximal energy. Andriantiana~\cite{andriantiana2011unicyclic} extended \cite{hou2002unicyclic}  by determining all unicyclic bipartite graphs of order $n>15$ having the maximal energy. Andriantiana and Wagner~\cite{andriantianaWagner2011unicyclic} also found the $n$-vertex non-bipartite unicyclic graphs with the maximum energy. Pedersen and Vestergaard~\cite{pedersen2005number} determined sharp lower and upper bounds for the Merrifield-Simmons index in a unicyclic graph of order $n$. They also found the unicyclic graphs that are maximal with respect to the Merrifield-Simmons index, given order and girth. Ou~\cite{ou2009extremal} characterised the unicyclic graphs of order $n$ having the largest as well as the second-largest Hosoya index. Zhu and Chen~\cite{zhu2011merrifield} determined the unicyclic graphs with prescribed girth and number of pendant vertices that have the maximal Merrifield-Simmons index.

Among unicyclic graphs with prescribed order (and potentially other structural restrictions), the largest and second-largest energies are usually attained by cycles and the so-called tadpole graphs (obtained by merging a vertex of a cycle with a pendant vertex of a path). Among all unicyclic graphs of order $n\geq 6$, the minimum energy is attained by the graph that results from connecting two leaves of a star by an edge. These extremal graphs will also play an important role in our current context of determining the number of connected induced subgraphs of a unicyclic graph of order $n$.

\medskip
The results in this paper can be generalised to arbitrary graphs with prescribed order and number of connected components (see Section~\ref{conclude}). In Section~\ref{conclude}, we also suggest further questions on the number of connected induced subgraphs for future investigation. Shortly before a first version of this paper appeared online, a paper by Pandey and Patra~\cite{Pandey2018} which studies the number of connected (not necessarily induced) subgraphs of graphs and unicyclic graphs, was posted on the arXiv. The extremal graphs in this paper turn out to be the same as those found in~\cite{Pandey2018} for large enough order.

\section{Preliminaries}\label{Sec:NkG}

Both the edgeless graph $E_n$ and the complete graph $K_n$ remain extremal for the number of connected induced subgraphs among all graphs of order $n$.

\begin{proposition}\label{connGraphs}	Every graph of order $n$ has at least $n$ connected induced subgraphs (with equality for the edgless graph $E_n$ only) and at most $2^n -1$ connected induced subgraphs (with equality for the complete graph $K_n$ only).
\end{proposition}

\begin{proof}
Every single vertex of a graph $G$ induces a connected subgraph of $G$. Clearly, the only connected induced subgraphs of the edgeless graph are its vertices. If $G$ is a graph of order $n$ which is not isomorphic to $E_n$, then $G$ has at least one edge $e$; so the endvertices of $e$ induce a $2$-vertex connected subgraph of $G$. This proves the minimisation part of the proposition. For the maximum, it is obvious that every induced subgraph of $K_n$ is connected; so $K_n$ has $2^n -1$ connected induced subgraphs. If $G$ is a graph of order $n$ which is not isomorphic to $K_n$, then $G$ has at least two nonadjacent vertices $u,v$: the subgraph induced by $\{u,v\}$ is not connected. This proves the proposition.
\end{proof}
		
Before we can continue, we need to introduce further terminologies and notations. For a connected graph $G$, we denote by $N_k(G)$ the number of $k$-vertex connected induced subgraphs of $G$, and by $N(G)$ the total number of connected induced subgraphs of $G$. By deleting a vertex $u \in V(G)$, we mean removing $u$ and all edges incident with $u$ in $G$. In particular, the subgraph induced by a (nonempty) set $W \subseteq V(G)$ is obtained by deleting in $G$ all vertices that do not belong to $W$. If $e \in E(G)$, then $G-e$ stands for the graph obtained from $G$ by removing edge $e$ in $G$ (but leaving the two endvertices of $e$).

From this point onwards, $G$ is always a connected graph. A very basic observation is that $N_1(G)=|V(G)|$, $N_2(G)=|E(G)|$ and $N_{|V(G)|}(G)=1$. It is important to note that all induced subgraphs (not necessarily connected) of order $|V(G)|-1$ of $G$ can be obtained by deleting one vertex from $G$, giving $N_{|V(G)|-1}(G) \leq |V(G)|$. One can even be more precise: a cut vertex of $G$ is a vertex of $G$ the deletion of which renders $G$ disconnected. Thus, $N_{|V(G)|-1} (G)$ is precisely the number of non-cut vertices of $G$. If $v_0,v_1,\ldots,v_{k-1} \in V(G)$, then we write $G-\{v_0,v_1,\ldots,v_{k-1}\}$ to mean the induced subgraph obtained from $G$ by deleting vertices $v_0,v_1,\ldots,v_{k-1}$. For $u,v \in V(G)$, we denote by $N(G)_{u}$ (resp. $N(G)_{u,v}$) the number of connected induced subgraphs of $G$ that contain $u$ (resp. both $u$ and $v$). The cycle of length $k \geq 3$ is denoted by $C_k=(v_0,v_1,\ldots,v_{k-1},v_k=v_0)$, where vertex $v_j$ is adjacent to vertex $v_{j+1}$ in $C_k$ for every $j \in \{0,1,\ldots,k-1\}$. 

\medskip
Our next theorem shows that the $n$-vertex path $P_n$ has very few connected $k$-vertex induced subgraphs, and this is in fact the only graph having the minimum number of connected $k$-vertex induced subgraphs among all connected graphs of order $n$. 

\begin{theorem}\label{MaxNk}
Every connected graph of order $n$ has at least $n-k+1$ connected induced subgraphs of order $k$, with equality holding (in the case $2< k<n$) only for the path $P_n$.
\end{theorem}

\begin{proof}
Let $G$ be a connected graph of order $n$. The cases $k\in \{1,2,n\}$ are trivial since $N_1(G)=n,~N_n(G)=1$ and $N_2(G) \geq n-1$ with equality if and only if $G$ is a tree. Assume that $2<k<n$ and that $G$ is not isomorphic to $P_n$. If $G$ is isomorphic to the cycle $C_n$, then removing an edge $e$ from $G$ yields the path $P_n=G-e$, while $G=C_n$ has a connected induced subgraph of every order $k$ that contains the edge $e$. Thus, we obtain $N_k (G)> N_k (P_n)$. If $G$ is not isomorphic to $C_n$, then $G$ has a spanning tree $T$ which is not isomorphic to $P_n$ and thus $N_k (G) \geq N_k (T)$. By Lemma~\ref{Jam:Lem}, we have $N_k (T) > N_k (P_n)=n-k+1$. This completes the proof of the theorem. 
\end{proof}

In particular, the path $P_n$ which has $\binom{n+1}{2}$ connected induced subgraphs, is the only graph that minimises the number of connected induced subgraphs among all connected graphs of order $n$. Note that the identity $N(P_n)=\binom{n+1}{2}$ was already mentioned in Theorem~\ref{Szek:Theo}. The maximum (analogue of Theorem~\ref{MaxNk}) can be attained by a rich class of connected graphs. We recall that $N_{|V(G)|-1}(G) \leq |V(G)|$. Equality holds if and only if $G$ is $2$-connected (the cycle and the complete graph, for instance).		

Denote by $\mathcal{G}^l_n$ the set of all nonisomorphic graphs that result from removing exactly $l \leq n/2$ independent edges (edges sharing no common vertex) in the complete graph $K_n~(n\geq 3)$. It is not difficult to see that every graph in $\mathcal{G}^l_n$ is of order $n$ and connected. In general, we have the following:

\begin{proposition}
For every connected graph $G$ of order $n\geq 3$ and every $k \in \{3,4,\ldots,n\}$, the number of $k$-vertex connected induced subgraphs of $G$ is at most $\binom{n}{k}$. Equality holds for all graphs $G \in \mathcal{G}^l_n$ for every $l \leq n/2$.
\end{proposition}

\section{The extremal unicyclic graphs of order $n$}\label{Sec:Uni}
In this section, we deal with lower and upper bounds on the number of connected induced subgraphs of a unicylic graph of order $n$, and also characterise all cases of equality. 

\subsection{The minimum number of connected induced subgraphs}

Before we can state our next theorem, we need to go through some preparation.

\begin{lemma}\label{FormCycle}
The cycle $C_n$ has $n^2 -n + 1$ connected induced subgraphs.
\end{lemma}

\begin{proof}
Let $C_n=(v_0,v_1,\ldots,v_{n-1},v_n=v_0)$ be the cycle of order $n$. Then a subset $S$ of $k$ elements of $V(C_n)$ induces a connected subgraph if and only if the vertices in $S$ can be arranged in the unique form $$
(v_j,v_{(j+1)\hspace*{-0.2cm}\mod n},v_{(j+2)\hspace*{-0.2cm}\mod n},\ldots,v_{(j+k-1)\hspace*{-0.2cm}\mod n})$$ for some $j \in \{0,1,\ldots,n-1\}$. This representation fails to be unique if and only if vertex $v_{(j+k-1)\hspace*{-0.1cm}\mod n}$ is adjacent to vertex $v_j$, which only happens when $k=n$. Therefore, $C_n$ has precisely $n$ connected induced subgraphs of every order $k$ between $1$ and $n-1$, while it has only one connected induced subgraph of order $n$ (the subgraph $C_n$ itself). 
\end{proof}

The tadpole graph $G_{p,q}$ is the connected graph obtained by identifying a vertex of the cycle $C_p$ with a pendant vertex of the path $P_{q+1}$. So the order of $G_{p,q}$ is $p+q$. 

\begin{lemma}\label{FormTrap}
The tadpole graph $G_{p,q}$ has
\begin{align*}
\binom{p}{2}+\binom{q+1}{2}+\frac{(q+1)(p^2-p+2)}{2}
\end{align*}
connected induced subgraphs.
\end{lemma}

\begin{proof}
Consider the tadpole graph $G_{p,q}$ as depicted in Figure~\ref{TadpoleGpq}. To evaluate $N(G_{p,q})$, we distinguish between subgraphs of $G_{p,q}$ that contain vertex $v_0$ and subgraphs of $G_{p,q}$ that do not contain $v_0$.
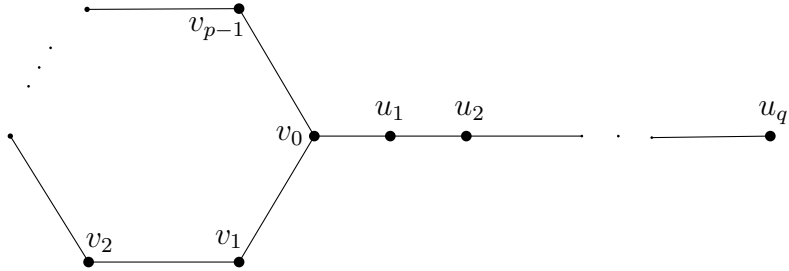
\begin{figure}[htbp]\centering
	\begin{tikzpicture}[line cap=round,line join=round,>=triangle 45,x=1.0cm,y=1.0cm]
	
	\draw (6.35,6.27) node[anchor=north west] {$v_0$};
	\draw (5.55,4.9) node[anchor=north west] {$v_1$};
	\draw (3.85,4.85) node[anchor=north west] {$v_2$};
	\draw (5.2,7.72) node[anchor=north west] {$v_{p-1}$};
	\draw (7,6)-- (8,6);
	\draw (8,6)-- (9,6);
	\draw (9,6)-- (10.52,6);
	\draw (11.44,5.98)-- (13,6);
	\draw (7.65,6.6) node[anchor=north west] {$u_1$};
	\draw (8.7,6.6) node[anchor=north west] {$u_2$};
	\draw (12.7,6.6) node[anchor=north west] {$u_q$};
	\draw (7,6)-- (6.01,7.69);
	\draw (7,6)-- (6.01,4.33);
	\draw (6.01,4.33)-- (4.03,4.33);
	\draw (4.03,4.33)-- (3,6);
	\draw (6.01,7.69)-- (4.01,7.68);
	
	\fill [color=black] (7,6) circle (2.0pt);
	\fill [color=black] (8,6) circle (2.0pt);
	\fill [color=black] (9,6) circle (2.0pt);
	\fill [color=black] (13,6) circle (2.0pt);
	\fill [color=black] (10.52,6) circle (0.5pt);
	\fill [color=black] (11.44,5.98) circle (0.5pt);
	\fill [color=black] (11,6) circle (0.5pt);
	\fill [color=black] (6.01,7.69) circle (2.0pt);
	\fill [color=black] (6.01,4.33) circle (2.0pt);
	\fill [color=black] (4.03,4.33) circle (2.0pt);
	\fill [color=black] (3,6) circle (1.0pt);
	\fill [color=black] (4.01,7.68) circle (1.0pt);
	\fill [color=black] (3.38,6.91) circle (0.5pt);
	\fill [color=black] (3.24,6.66) circle (0.5pt);
	\fill [color=black] (3.53,7.17) circle (0.5pt);
	\end{tikzpicture}
	\caption{The tadpole graph $G_{p,q}$.}\label{TadpoleGpq}
\end{figure}
Since deleting $v_0$ in $G_{p,q}$ yields the two connected components $P_{p-1}$ and $P_q$, we deduce by Theorem~\ref{MaxNk} that $$N(G_{p,q}-\{v_0\})=
N(P_{p-1})+N(P_q)=\binom{p}{2}+\binom{q+1}{2}.$$ On the other hand, every $v_0$-containing connected induced sugraph of $G_{p,q}$ is uniquely determined by merging a $v_0$-containing connected induced subgraph of $C_p$ and a $v_0$-containing connected induced subgraph of $P_{q+1}$ at vertex $v_0$ (a fixed pendant vertex of $P_{q+1}$). Thus, we have
$$
N(G_{p,q})_{v_0}=\frac{(q+1)(p^2-p+2)}{2}.
$$
Indeed, $P_{q+1}$ has exactly $q+1$ subtrees containing a fixed pendant vertex, while by the proof of Lemma~\ref{FormCycle}, we have $N(C_p)_{v_0}=p^2 -p + 1 - \binom{p}{2}$. This completes the proof of the lemma.
\end{proof}

We shall also need the following simple lemma about trees.

\begin{lemma}\label{rootsubt}
The number of root-containing subtrees of a rooted tree $T$ is at least the order of $T$, with equality if and only if $T$ is a path rooted at one of its pendant vertices.
\end{lemma}
	
\begin{proof}
Let $T$ be a tree with root $v$. For $u \in V(T)$, the unique path between $v$ and $u$ in $T$ is a subtree of $T$. Thus, $T$ has at least $|V(T)|$ subtrees containing $v$. If $T$ is not a path, or if $T$ is a path but $v$ is not a leaf of $T$, then the tree $T$ itself was not one of the $|V(T)|$ subtrees counted before. In this case, we get strict inequality, that is $N(T)_v> |V(T)|$.	
\end{proof}

Our next result shows that the tadpole graph $G_{3,n-3}$ is the unique unicyclic graph of order $n$ with the minimum number of connected induced subgraphs. It is interesting to point out that the tadpole graph $G_{3,n-3}$ is also known to maximise the Wiener index (the sum of the distances between all unordered pairs of vertices) among all unicyclic graphs of a given order; see~\cite{hong2011wiener,tang2008n,yu2010wiener}. 

\begin{theorem}\label{MinforUnic}
If $G$ is a unicylic graph of order $n\geq 3$, then $N(G) \geq (n^2 + 3n - 4)/2 $, with equality if and only if $G$ is isomorphic to $G_{3,n-3}$.	
\end{theorem} 

\begin{proof}
By Lemma~\ref{FormTrap}, we have $$
N(G_{3,n-3})=\frac{(n-1)(n+4)}{2}=\frac{n^2 + 3n - 4}{2}\,.
$$ The statement is seen to hold for $n=3$ as $C_3=G_{3,0}$ is the only unicyclic graph of order $3$, and $N(C_3)=7$. By Lemma~\ref{FormCycle}, we have $N(C_n)=n^2-n+1$ and so it is easy to see that $$
n^2-n+1=N(C_n) > \frac{n^2 + 3n - 4}{2} =N(G_{3,n-3})
$$ provided that $n \neq 3$. For $n>3$, let $G \neq C_n$ be a unicyclic graph of order $n$. It is clear that $G$ has at least one pendant vertex (otherwise, $G$ is a cycle). Let $v$ be a pendant vertex of $G$ so that $G-\{v\}$ is a unicyclic graph of order $n-1$. We then induct on $n$ to prove that $N(G)\geq (n-1)(n+4)/2$ with equality holding only for $G_{3,n-3}$.

Consider the unique cycle $C_k=(v_0,v_1,v_2,\ldots, v_{k-1}, v_k=v_0)$ of $G$. For every $j \in \{0,1,2,\ldots,k-1\}$, the subgraph $T_j$ of $G$ depicted in Figure~\ref{ShapeUnic} is a tree rooted at vertex $v_j$.
\begin{figure}[htbp]\centering
\begin{tikzpicture}[line cap=round,line join=round,>=triangle 45,x=1.0cm,y=1.0cm]
\draw (8.38,8.27) node[anchor=north west] {$v_0$};
\draw (7.57,6.9) node[anchor=north west] {$v_1$};
\draw (5.83,6.84) node[anchor=north west] {$v_2$};
\draw (7.22,9.68) node[anchor=north west] {$v_{k-1}$};
\draw (9.3,7.86) node[anchor=north west] {$T_0$};
\draw (8.2,6.1) node[anchor=north west] {$T_1$};
\draw (4.5,6.52) node[anchor=north west] {$T_2$};
\draw (8.2,10.6) node[anchor=north west] {$T_{k-1}$};
\draw (9,8)-- (8,6.36);
\draw (6.01,6.37)-- (8,6.36);
\draw (6.01,6.37)-- (5,8);
\draw (9,8)-- (8.03,9.67);
\draw (8.03,9.67)-- (5.91,9.66);
\draw [rotate around={-88.78:(8.01,10.33)},dash pattern=on 1pt off 1pt] (8.01,10.33) ellipse (0.7cm and 0.2cm);
\draw [rotate around={0.07:(9.64,8)},dash pattern=on 1pt off 1pt] (9.64,8) ellipse (0.66cm and 0.18cm);
\draw [rotate around={-89.87:(8,5.68)},dash pattern=on 1pt off 1pt] (8,5.68) ellipse (0.72cm and 0.22cm);
\draw [rotate around={36.22:(5.5,6)},dash pattern=on 1pt off 1pt] (5.5,6) ellipse (0.67cm and 0.24cm);

\fill [color=black] (9,8) circle (2.0pt);
\fill [color=black] (8,6.36) circle (2.0pt);
\fill [color=black] (6.01,6.37) circle (2.0pt);
\fill [color=black] (5,8) circle (0.5pt);
\fill [color=black] (8.03,9.67) circle (2.0pt);
\fill [color=black] (5.91,9.66) circle (0.5pt);
\fill [color=black] (5.25,9.01) circle (0.5pt);
\fill [color=black] (5.45,9.34) circle (0.5pt);
\fill [color=black] (5.03,8.73) circle (0.5pt);
\fill [color=black] (8,11) circle (0.5pt);
\fill [color=black] (10.28,8) circle (0.5pt);
\fill [color=black] (8,5) circle (0.5pt);
\fill [color=black] (4.99,5.62) circle (0.5pt);
\fill [color=black] (8.14,10.85) circle (0.5pt);
\fill [color=black] (10.15,7.88) circle (0.5pt);
\fill [color=black] (7.84,5.17) circle (0.5pt);
\fill [color=black] (5.16,5.57) circle (0.5pt);
\end{tikzpicture}
\caption{The general shape of a (connected) unicyclic graph.}\label{ShapeUnic}
\end{figure}
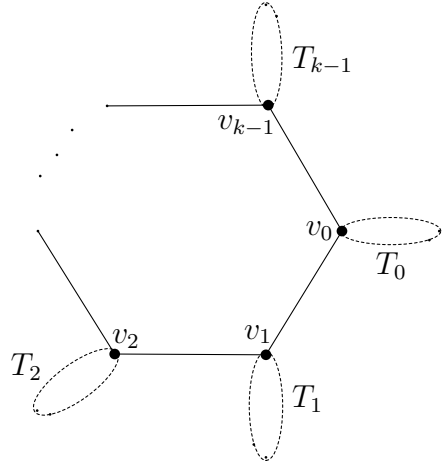	
We can assume (without loss of generality) that $v \neq v_0$ is a leaf of $T_0$. Denote by $a_j$ the number of $v_j$-containing subtrees of $T_j$ and by $S_{j,1},S_{j,2},\ldots,S_{j,a_j} $ all the corresponding $a_j$ subtrees. Let $S_{0,1},S_{0,2},\ldots,S_{0,b_0}$ be those subtrees of $T_0$ that contain both $v$ and $v_0$. It follows immediately that the subgraphs induced by 
\begin{align}\label{Essent}
\begin{split}
&V(S_{0,1}), V(S_{0,2}),\ldots, V(S_{0,b_0})\,, \\
&V(T_0) \cup V(S_{1,1}), V(T_0) \cup V(S_{1,2}), \ldots, V(T_0) \cup V(S_{1,a_1}), \\ 
& V(T_0) \cup V(T_1) \cup V(S_{2,1}), V(T_0) \cup V(T_1) \cup V(S_{2,2}), \ldots,\\
& V(T_0) \cup V(T_1) \cup \cdots \cup V(T_{k-2}) \cup V(S_{k-1,1}), V(T_0) \cup V(T_1) \cup \cdots \cup V(T_{k-2}) \cup V(S_{k-1,2}), \\
& \ldots, V(T_0) \cup V(T_1) \cup \cdots \cup V(T_{k-2}) \cup V(S_{k-1,a_{k-1}})
\end{split}
\end{align}
in $G$ are all connected induced subgraphs of $G$ that contain both $v$ and $v_0$. Thus, their number is precisely $b_0+a_1+\cdots +a_{k-1}$. Since $k>2$, vertices $v_1$ and $v_{k-1}$ are distinct; so the subgraph induced by $V(T_0) \cup V(S_{k-1,1})$ also contains both $v$ and $v_0$.

Now consider those subtrees of $T_0$ that contain $v$ but not $v_0$. Their number is $N(T_0)_v - b_0$, implying that the number of connected induced subgraphs of $G$ that contain $v$ is at least $$
b_0+a_1+\cdots +a_{k-1} +1 + N(T_0)_v - b_0.$$ By Lemma~\ref{rootsubt}, we have $N(T_0)_v \geq |V(T_0)|$ and $a_j \geq |V(T_j)|$ for every $j \in \{0,1,\ldots,k-1 \}$. Therefore, we get
\begin{align}\label{OneDist}
N(G)_v \geq |V(T_1)|+\cdots + |V(T_{k-1})| +1 + |V(T_0)| =n+1.
\end{align}
Suppose $k>3$. Then the subgraph induced by $V(T_0) \cup V(S_{k-1,1}) \cup V(S_{k-2,1})$ is connected and its vertex set is distinct from $V(T_0) \cup V(S_{k-1,1})$ as well as all the sets enumerated in~\eqref{Essent}. Consequently, by the above discussion, equality holds in~\eqref{OneDist} if and only if $$k=3,~a_1=a_2=1,~T_0=P_{|V(T_0)|}~\text{ and }~v_0~\text{is a pendant vertex of}~T_0.$$

On the other hand, since $G-\{v\}$ is a unicyclic graph of order $n-1$, we obtain
$$ N(G-\{v\}) \geq  \frac{(n-2)(n+3)}{2}$$ by the induction hypothesis, with equality holding if and only if $G-\{v\}$ is the tadpole graph $G_{3,n-4}$. Summing up the two contributions yields
$$
N(G)= N(G)_v+  N(G-\{v\}) \geq  n+1 + \frac{(n-2)(n+3)}{2}  = \frac{(n-1)(n+4)}{2}.
$$ Equality holds if and only if $G-\{v\}=G_{3,n-4},~k=3,~|V(T_1)|=|V(T_2)|=1,~ T_0=P_{n-2}~\text{and}~v_0~\text{is a pendant vertex of}~T_0$. In this case, $G$ is indeed isomorphic to $G_{3,n-3}$, completing the proof of the theorem.
\end{proof}

\subsection{The maximum number of connected induced subgraphs}
Before getting to our next and final theorem (the analogue of Theorem~\ref{MinforUnic}), we need to start with a few definitions and intermediate results.

Recall that the star $S_n$ is obtained by attaching $n-1$ pendant vertices to a single vertex, called the central vertex. Denote by $Q_n$ the connected graph obtained by adding one edge between two pendant vertices of the star $S_n,~n>2$. Then $Q_n$ contains only one cycle and its length is $3$. We are going to show that $Q_n$ is maximal with respect to the number of connected induced subgraphs of a unicyclic graph of order $n$. Again, it is worth mentioning that the graph $Q_n$ is known to minimise the Wiener index among all unicyclic graphs of a given order~\cite{hong2011wiener,tang2008n,yu2010wiener}.

\begin{theorem}\label{MaxUnic}
If $G$ is a unicyclic graph of order $n>5$, then $N(G)\leq n+2^{n-1}$, with equality if and only if $G$ is isomorphic to $Q_n$.
\end{theorem} 

It is important to note that since $Q_4=G_{3,1}$ and there are only two unicyclic graphs of order $4$, the cycle $C_4$ is therefore the unicyclic graph of order $4$ with the greatest number of connected induced subgraphs. On the other hand, it is not difficult to see that out of the five unicyclic graphs of order $5$ (see Figure~\ref{FigUni5}), the cycle $C_5$, the graph $Q_5$ and the so-called banner graph $B_5$ (obtained by attaching a pendant edge to a vertex of $C_4$) all have the maximum number of connected induced subgraphs, namely $21$.
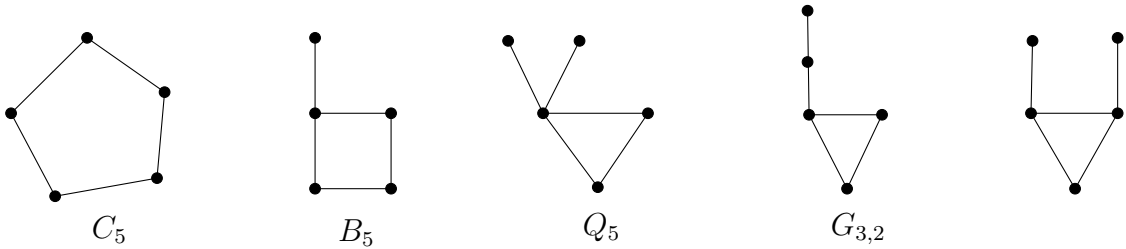
\begin{figure}[htbp]\centering
\begin{tikzpicture}[line cap=round,line join=round,>=triangle 45,x=1.0cm,y=1.0cm]
\draw (1.,6.)-- (2.,7.);
\draw (2.,7.)-- (3.02,6.28);
\draw (3.02,6.28)-- (2.92,5.14);
\draw (2.92,5.14)-- (1.58,4.9);
\draw (1.58,4.9)-- (1.,6.);
\draw (5.,7.)-- (5.,6.);
\draw (5.,6.)-- (6.,6.);
\draw (6.,6.)-- (6.,5.);
\draw (6.,5.)-- (5.,5.);
\draw (5.,5.)-- (5.,6.);
\draw (8.72,5.02)-- (8.,6.);
\draw (8.,6.)-- (9.38,6.);
\draw (9.38,6.)-- (8.72,5.02);
\draw (8.,6.)-- (7.54,6.96);
\draw (8.,6.)-- (8.48,6.96);
\draw (12.,5.)-- (11.5,5.98);
\draw (12.,5.)-- (12.46,5.98);
\draw (11.5,5.98)-- (12.46,5.98);
\draw (11.5,5.98)-- (11.48,7.36);
\draw (15.,5.)-- (14.42,6.);
\draw (15.,5.)-- (15.56,6.);
\draw (14.42,6.)-- (15.56,6.);
\draw (14.42,6.)-- (14.44,6.96);
\draw (15.56,6.)-- (15.56,7.);
\draw (1.92,4.78) node[anchor=north west] {$C_5$};
\draw (5.16,4.74) node[anchor=north west] {$B_5$};
\draw (8.38,4.82) node[anchor=north west] {$Q_5$};
\draw (11.64,4.82) node[anchor=north west] {$G_{3,2}$};

\draw [fill=black] (2.,7.) circle (2.0pt);
\draw [fill=black] (1.,6.) circle (2.0pt);
\draw [fill=black] (3.02,6.28) circle (2.0pt);
\draw [fill=black] (1.58,4.9) circle (2.0pt);
\draw [fill=black] (2.92,5.14) circle (2.0pt);
\draw [fill=black] (5.,7.) circle (2.0pt);
\draw [fill=black] (5.,6.) circle (2.0pt);
\draw [fill=black] (6.,6.) circle (2.0pt);
\draw [fill=black] (5.,5.) circle (2.0pt);
\draw [fill=black] (6.,5.) circle (2.0pt);
\draw [fill=black] (8.,6.) circle (2.0pt);
\draw [fill=black] (8.72,5.02) circle (2.0pt);
\draw [fill=black] (9.38,6.) circle (2.0pt);
\draw [fill=black] (7.54,6.96) circle (2.0pt);
\draw [fill=black] (8.48,6.96) circle (2.0pt);
\draw [fill=black] (12.,5.) circle (2.0pt);
\draw [fill=black] (11.5,5.98) circle (2.0pt);
\draw [fill=black] (12.46,5.98) circle (2.0pt);
\draw [fill=black] (11.48,6.68) circle (2.0pt);
\draw [fill=black] (11.48,7.36) circle (2.0pt);
\draw [fill=black] (15.,5.) circle (2.0pt);
\draw [fill=black] (14.42,6.) circle (2.0pt);
\draw [fill=black] (15.56,6.) circle (2.0pt);
\draw [fill=black] (14.44,6.96) circle (2.0pt);
\draw [fill=black] (15.56,7.) circle (2.0pt);
\end{tikzpicture}
\caption{All the unicyclic graphs of order $5$.}\label{FigUni5}
\end{figure}

We begin with a counterpart of Lemma~\ref{rootsubt}.

\begin{proposition}\label{upperTr}
Let $T$ be a rooted tree of order $n$ whose root is $v$. Then the number of $v$-containing subtrees of $T$ is at most $2^{n-1}$. Equality holds if and only if $T$ is the star $S_n$ and $v$ is the center of $T$.
\end{proposition}

\begin{proof}
There are $2^{n-1}$ subsets of vertices of $T$ that contain $v$. Thus $N(T)_v \leq 2^{n-1}$. If $T$ is a star whose center is $v$, then all $2^{n-1}$ such subsets induce a connected subgraph of $T$. Otherwise, there is a vertex $u$ of $T$ at distance at least $2$ from $v$. In this case, the set $\{u, v\}$ does not induce a connected subgraph of $T$.	
\end{proof}

In general, we define the banner graph $B_n$ of order $n\geq 4$ to be the graph constructed from $C_4$ by adding $n-4$ pendant edges to the same vertex of $C_4$.

\begin{lemma}\label{FormBanner}
The banner graph $B_n$ has $2+n+7\cdot 2^{n-4}$ connected induced subgraphs.
\end{lemma}

\begin{proof}
The statement is seen to hold for $n=4$ as $B_4=C_4$ and $N(C_4)=13$ by Lemma~\ref{FormCycle}. For $n>4$, let $v$ be the neighbor of a pendant vertex of $B_n$ (see Figure~\ref{GraphBn}). Deleting $v$ in $B_n$ yields $n-4$ copies of $P_1$ and one copy of $P_3$. So $N(B_n-\{v\})= n+2$ as $N(P_m)=\binom{m+1}{2}$ by Theorem~\ref{MaxNk}.
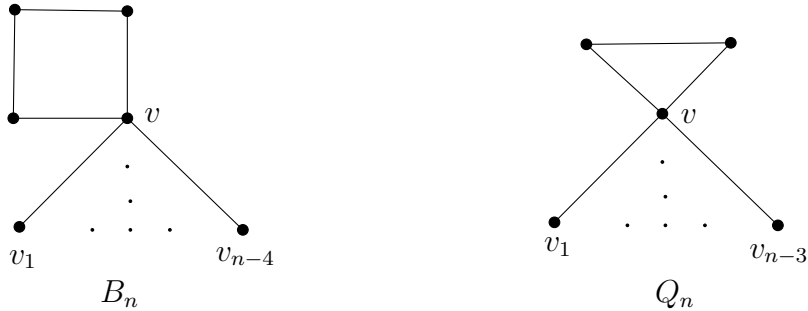
\begin{figure}[htbp]\centering
\begin{tikzpicture}[line cap=round,line join=round,>=triangle 45,x=1.0cm,y=1.0cm]
\draw (3.02,8.44)-- (4.5,8.42);
\draw (4.5,8.42)-- (4.5,7.);
\draw (4.5,7.)-- (3.,7.);
\draw (3.,7.)-- (3.02,8.44);
\draw (4.5,7.)-- (3.08,5.56);
\draw (4.5,7.)-- (6.02,5.52);
\draw (4.58,7.28) node[anchor=north west] {$v$};
\draw (2.8,5.4) node[anchor=north west] {$v_1$};
\draw (5.52,5.44) node[anchor=north west] {$v_{n-4}$};
\draw (4.,5) node[anchor=north west] {$B_n$};

\draw (10.54,7.98)-- (12.44,8.);
\draw (12.44,8.)-- (11.54,7.06);
\draw (11.54,7.06)-- (10.12,5.62);
\draw (11.54,7.06)-- (13.06,5.58);
\draw (11.64,7.27) node[anchor=north west] {$v$};
\draw (9.84,5.56) node[anchor=north west] {$v_1$};
\draw (12.54,5.48) node[anchor=north west] {$v_{n-3}$};
\draw (11.54,7.06)-- (10.54,7.98);

\draw (11.3,5) node[anchor=north west] {$Q_n$};
\draw [fill=black] (3.02,8.44) circle (2.0pt);
\draw [fill=black] (4.5,8.42) circle (2.0pt);
\draw [fill=black] (3.,7.) circle (2.0pt);
\draw [fill=black] (4.5,7.) circle (2.0pt);
\draw [fill=black] (3.08,5.56) circle (2.0pt);
\draw [fill=black] (6.02,5.52) circle (2.0pt);
\draw [fill=black] (4.04,5.52) circle (0.5pt);
\draw [fill=black] (4.54,5.52) circle (0.5pt);
\draw [fill=black] (5.06,5.52) circle (0.5pt);
\draw [fill=black] (4.5,6.36) circle (0.5pt);
\draw [fill=black] (4.54,5.9) circle (0.5pt);
\draw [fill=black] (10.54,7.98) circle (2.0pt);
\draw [fill=black] (12.44,8.) circle (2.0pt);
\draw [fill=black] (11.54,7.06) circle (2.0pt);
\draw [fill=black] (10.12,5.62) circle (2.0pt);
\draw [fill=black] (13.06,5.58) circle (2.0pt);
\draw [fill=black] (11.08,5.58) circle (0.5pt);
\draw [fill=black] (11.58,5.58) circle (0.5pt);
\draw [fill=black] (12.1,5.58) circle (0.5pt);
\draw [fill=black] (11.54,6.42) circle (0.5pt);
\draw [fill=black] (11.58,5.96) circle (0.5pt);
\end{tikzpicture}
\caption{The unicyclic graphs $B_n$ (left) and $Q_n$ (right) of order $n$.}\label{GraphBn}
\end{figure}
On the other hand, every $v$-containing connected induced subgraph of $B_n$ decomposes naturally into a $v$-containing connected induced subgraph of $C_4$ and a $v$-containing connected induced subgraph of $S_{n-3}$ (whose center is $v$). This gives $N(B_n)_v = N(C_4)_v \cdot N(S_n)_v = 7 \cdot 2 ^{n-4}$ by Proposition~\ref{upperTr}, completing the proof.
\end{proof}

To prove Theorem~\ref{MaxUnic}, we shall need one more intermediate result (Proposition~\ref{upperTlr}). In a certain sense, this proposition parallels a result of Sz{\'e}kely and Wang~\cite{szekely2005subtrees}, who proved that the star maximises the number of subtrees that contain at least one leaf of $T$ among all trees $T$ of a given order.

\begin{proposition}\label{upperTlr}
Let $T$ be a rooted tree of order $n>1$ whose root is $v$. If $l \neq v$ is a leaf of $T$, then the number of subtrees of $T$ that contain both $v$ and $l$ is at most $2^{n-2}$. Equality holds if and only if $T$ is the star $S_n$ and $v$ is its center.
\end{proposition}

\begin{proof}
The statement is trivial for $n=2$. Assume that $n>2$. There are precisely $2^{n-2}$ subsets of vertices of $T$ that contain $v$ and $l$, showing that $N(T)_{v,l}\leq 2^{n-2}$. If $T$ is a star whose center is $v$, then all such subsets of vertices induce a connected subgraph, meaning that $N(S_n)_{v,l}= 2^{n-2}$. If $T$ is a star but $v$ is not its center, then the graph induced by $\{v,l\}$ is not connected. It remains to rule out the case where $T$ is not a star.

Assume that $T$ is not a star. Then there exists a vertex $w \in V(T)$ at distance $2$ from $v$. If $l$ is not adjacent to $v$, then the graph induced by $\{v,l\}$ is not connected. Otherwise, $l$ and $v$ are adjacent, and since $l$ is a leaf, the graph induced by $\{v,l,w\}$ is not connected. This completes the proof of the proposition.
\end{proof}

We are now ready to give a proof of the main result (Theorem~\ref{MaxUnic}).
\begin{proof}[Proof of Theorem~\ref{MaxUnic}]
For $n>3$, let $v$ be the unique vertex of $Q_n$ whose degree is at least $3$ (see Figure~\ref{GraphBn}). Fix a vertex $v_1$ of $P_2$ and a vertex $v_2$ of $C_3$. We can then construct all $v$-containing connected induced subgraphs of $Q_n$ as follows: take $n-3$ connected induced subgraphs of $P_2$ that contain $v_1$, one connected induced subgraph of $C_3$ that contains $v_2$, and identify both $v_1$ and $v_2$ with $v$. This gives us $N(Q_n)_v=4\cdot 2^{n-3}=2^{n-1}$. On the other hand, deleting $v$ in $Q_n$ yields $n-3$ copies of $P_1$ and one copy of $P_2$, meaning that $N(Q_n-\{v\})=n$. This proves the identity $N(Q_n)=2^{n-1}+n$.

By Lemma~\ref{FormBanner}, it is obvious that $N(B_n)=2+n + 7 \cdot 2^{n-4} < n+2^{n-1}=N(Q_n)$ if $n>5$. Let $G$ be a graph having the maximum number of connected induced subgraphs among all unicyclic graphs of order $n> 5$. For the rest of the proof, we then assume that $G$ is not a banner graph. Let $C_k=(v_0,v_1,\ldots,v_{k-1},v_k=v_0)$ be the unique cycle of $G$. Then $G$ has precisely the shape depicted in Figure~\ref{ShapeUnic}, where $T_0,T_1,\ldots,T_{k-1}$ are all trees rooted at vertices $v_0,v_1,\ldots,v_{k-1}$, respectively. 

\textbf{Claim 1}: \textit{For all $j \in \{0,1,\ldots,k-1\}$, the tree $T_j$ is a star with center $v_j$.}

For the proof of the claim, suppose (without loss of generality) that $T_0$ is not a star, or that $T_0$ is a star but $v_0$ is not its center. Let us first derive a formula for the number $N(G;T_0)$ of connected induced subgraphs of $G$ that contain a vertex of $T_0$. We partition these subgraphs into two categories:
\begin{itemize}
\item those that are induced entirely by a subset of vertices of $T_0$,
\item those that contain at least a vertex in $V(G)-V(T_0)$.
\end{itemize}
Let $M$ denote the number of connected induced subgraphs of $G-(V(T_0)-\{v_0\})$ that contain $v_0$ and at least one other vertex. Then the number of connected induced subgraphs of $G$ that contain a vertex of $T_0$ and at least a vertex in $V(G)-V(T_0)$ is $M \cdot N(T_0)_{v_0}$. Therefore, we have
$$
N(G;T_0)=N(T_0) + M \cdot N(T_0)_{v_0}\,.
$$
Construct from $G$ a new unicyclic graph $G^{\prime}$ by replacing $T_0$ with the star $T_0^{\prime}=S_{|V(T_0)|}$ centered at vertex $v_0$. Thus, we have
$$
N(G^{\prime}; T_0^{\prime})=N(T_0^{\prime}) +  M \cdot N(T_0^{\prime})_{v_0}
$$ as $G-(V(T_0)-\{v_0\})$ and $G^{\prime}-(V(T_0^{\prime})-\{v_0\})$ are isomorphic. The difference gives
$$
N(G^{\prime}; T_0^{\prime}) - N(G;T_0) =N(T_0^{\prime}) - N(T_0) + M (N(T_0^{\prime})_{v_0} -  N(T_0)_{v_0}).
$$
By Proposition~\ref{upperTr}, we have $N(T_0^{\prime})_{v_0}- N(T_0)_{v_0}  >0$, while it is also known that $N(T_0^{\prime})- N(T_0) \geq  0$; see Theorem~\ref{Szek:Theo}. Hence, we get $N(G^{\prime}; T_0^{\prime}) >  N(G;T_0)$ as $M>0$ by definition. Note that $G - V(T_0)$ and $G^{\prime} - V(T_0^{\prime})$ are isomorphic, so they have the same number of connected induced subgraphs. We then deduce that $N(G^{\prime}) > N(G)$, a contradiction. 

\medskip
In the following, we assume that each of the trees $T_0,T_1,\ldots,T_{k-1}$ is a star centered at vertices $v_0,v_1,\ldots,v_{k-1}$, respectively. We can also assume that $T_0$ has the maximum order among the trees $T_0,T_1,\ldots,T_{k-1}$. 

\textbf{Claim 2}: \textit{We have $k=3$}.

For the proof of this claim, suppose (for contradiction) that $k>3$. Let us first establish a formula for the number $N(G; T_0 \cup T_{k-1})$ of connected induced subgraphs of $G$ that contain a vertex in $V(T_0) \cup V(T_{k-1})$. To this end, we introduce the following notation:
\begin{itemize}
\item $M_1$ is the number of connected induced subgraphs of $G-\big( V(T_{k-1}) \cup V(T_0) -\{v_0\}\big)$ that contain $v_0$ and at least one other vertex;
\item $M_2$ is the number of connected induced subgraphs of $G-\big( V(T_0) \cup V(T_{k-1})-\{v_{k-1}\} \big)$ that contain $v_{k-1}$ and at least one other vertex;
\item $M_3$ is the number of connected induced subgraphs of $G-\big( V(T_0) \cup V(T_{k-1})-\{v_{k-1}, v_0\} \big)$ that contain both $v_0$ and $v_{k-1}$.
\end{itemize}
With this notation, the number of connected induced subgraphs of $G$ that contain
\begin{enumerate}[(i)]
\item a vertex of $T_0$ but no vertex of $T_{k-1}$ is given by
$$
N(T_0)  + M_1 \cdot N(T_0)_{v_0};$$
\item a vertex of $T_{k-1}$ but no vertex of $T_0$ is given by
$$
N(T_{k-1}) ~ + M_2 \cdot N(T_{k-1})_{v_{k-1}};$$
\item a vertex of $T_0$ and a vertex of $T_{k-1}$ is given by
$$
M_3 \cdot  N(T_0)_{v_0} \cdot  N(T_{k-1})_{v_{k-1}}.$$
\end{enumerate}
Combining all three cases, we obtain 
\begin{align}\label{OneG}
\begin{split}
 N(G; T_0 \cup T_{k-1})& =N(T_0) + N(T_{k-1}) + M_1 \cdot N(T_0)_{v_0}  + M_2 \cdot N(T_{k-1})_{v_{k-1}} \\
 &+  M_3 \cdot N(T_0)_{v_0} \cdot N(T_{k-1})_{v_{k-1}}\,.
 \end{split}
\end{align}

Construct from $G$ a new unicyclic graph $G^{\prime}$ by deleting all vertices of $T_{k-1}$ except $v_{k-1}$, then contracting vertex $v_{k-1}$, and finally replacing $T_0$ with the star $T_0^{\prime}=S_{|V(T_0)|+|V(T_{k-1})|}$ centered at vertex $v_0$. Denote by $M^{\prime}$ the number of connected induced subgraphs of $G^{\prime}-(V(T_0^{\prime})-\{v_0\})$ that contain $v_0$ and at least one other vertex. Thus, we obtain
\begin{align}\label{OneGPrimePrime}
 N(G^{\prime}; T_0^{\prime})  =N(T_0^{\prime}) + M^{\prime} \cdot N(T_0^{ \prime})_{v_0}
\end{align}
for the number $ N(G^{\prime}; T_0^{\prime})$ of connected induced subgraphs of $G^{\prime}$ that contain a vertex of $T_0^{\prime}$. Note that the quantities $M_1,M_2,M_3, M^{\prime}$ are explicitly given by 
\begin{align*}
M^{\prime}&=M_3 -1\,,\\
M_1&=\sum_{l=1}^{k-2} \prod_{j=1}^l N(T_j)_{v_j}\,,~ M_2=\sum_{r=2}^{k-1} \prod_{j=2}^r N(T_{k-j})_{v_{k-j}}\,,\\
M_3&=1+  M_1 + \sum_{r=2}^{k-2} \prod_{j=2}^r N(T_{k-j})_{v_{k-j}} + \sum_{l=1}^{k-4}~\sum_{r=2}^{k-l-2} \prod_{j=1}^l N(T_j)_{v_j}~\prod_{j=2}^r N(T_{k-j})_{v_{k-j}}\,.
\end{align*} 
Therefore, the difference~\eqref{OneGPrimePrime}-\eqref{OneG} gives
\begin{align}\label{DifferencePrimeP}
\begin{split}
&N(G^{\prime};  T_0^{\prime})  - N(G;  T_0 \cup   T_{k-1})=N(T_0^{\prime}) -N(T_0) - N(T_{k-1}) - N(T_0)_{v_0} N(T_{k-1})_{v_{k-1}}\\
& + \big(N(T_0^{\prime })_{v_0} -   N(T_0)_{v_0} (1+ N(T_{k-1})_{v_{k-1}}) \big) M_1\\
& + \big(N(T_0^{\prime})_{v_0} - N(T_{k-1})_{v_{k-1}} (1+ N(T_0)_{v_0} ) \big) \sum_{r=2}^{k-2} \prod_{j=2}^r N(T_{k-j})_{v_{k-j}}\\
& + (N(T_0^{\prime })_{v_0} -  N(T_0)_{v_0} N(T_{k-1})_{v_{k-1}})\sum_{l=1}^{k-5} \sum_{r=2}^{k-l-2}  \prod_{j=1}^l N(T_j)_{v_j}  ~ \prod_{j=2}^r N(T_{k-j})_{v_{k-j}} \\
& + (N(T_0^{\prime })_{v_0} -  N(T_0)_{v_0} N(T_{k-1})_{v_{k-1}} - N(T_{k-3})_{v_{k-3}} N(T_{k-1})_{v_{k-1}} ) \prod_{\substack{j=1 \\ j \neq k-3}}^{k-2} N(T_j)_{v_j}
\end{split}
\end{align}
after some basic manipulations. To simplify notation, we set $n_j=|V(T_j)|$ for every $j \in \{0,1,\ldots,k-1\}$.
	
By Proposition~\ref{upperTr} and Theorem~\ref{Szek:Theo}, we have
\begin{align*}
 N(T_0^{\prime}) & -  N(T_0) -  N(T_{k-1}) - N(T_0)_{v_0} \cdot N(T_{k-1})_{v_{k-1}} \\
&=(2^{n_0+n_{k-1} -1}+n_0+n_{k-1} -1 ) - (2^{n_0-1} +n_0 -1 )\\
& - (2^{n_{k-1} -1}+n_{k-1}  -1 ) - 2^{n_0-1}\cdot 2^{n_{k-1}-1}\\
& = (2^{n_{k-1}-1} -1) (2^{n_0-1} -1 ) \geq 0
\end{align*}
and
\begin{align*}
N(T_0^{\prime })_{v_0} - N(T_0)_{v_0} \big(1+ N(T_{k-1})_{v_{k-1}} \big)&=2^{n_0+n_{k-1} -1} - 2^{n_0-1}(1+ 2^{n_{k-1}-1} )\\
&=2^{n_0 -1} (2^{n_{k-1}-1} - 1 ) \geq 0\,.
\end{align*}
Likewise,
\begin{align*}
N(T_0^{\prime })_{v_0} - N(T_{k-1})_{v_{k-1}} \big(1+ N(T_0)_{v_0}  \big)&=2^{n_0+n_{k-1} -1} - 2^{n_{k-1}-1} (1+ 2^{n_0-1} )\\
&=2^{n_{k-1} -1} (2^{n_0-1} - 1) \geq 0
\end{align*}
and 
\begin{align*}
N(T_0^{\prime})_{v_0} - N(T_0)_{v_0} N(T_{k-1})_{v_{k-1}}&=2^{n_0+n_{k-1} -1} - 2^{n_0-1} \cdot 2^{n_{k-1} -1}\\
&=2^{n_0+n_{k-1} -2} > 0\,.
\end{align*}
Also, we have  
\begin{align*}
N(T_0^{\prime })_{v_0} -  N(T_0)_{v_0} & N(T_{k-1})_{v_{k-1}} - N(T_{k-3})_{v_{k-3}} N(T_{k-1})_{v_{k-1}}\\
&=2^{n_{k-1} -2} (2^{n_0}-2^{n_{k-3}}) \geq 0
\end{align*}
as $n_0 \geq n_{k-3}$ by assumption. The following conclusions about~\eqref{DifferencePrimeP} can be derived immediately:
\begin{itemize}
\item If $k\geq 6$, then $N(G^{\prime };  T_0^{\prime})  >  N(G;  T_0 \cup   T_{k-1})$;
\item If $k=5$, then
\begin{align*}
0 \leq N(G^{\prime };  T_0^{\prime })  - & N(G;  T_0 \cup   T_4)=(2^{n_4-1} -1) (2^{n_0-1} -1 )\\
&  + 2^{n_0 -1} (2^{n_4-1} - 1) \sum_{l=1}^3 \prod_{j=1}^l N(T_j)_{v_j}\\
&+ 2^{n_4 -1} (2^{n_0-1} - 1) \sum_{r=2}^3 \prod_{j=2}^r N(T_{5-j})_{v_{5-j}}\\
& + 2^{n_4 -2} (2^{n_0}-2^{n_2}) N(T_1)_{v_1}  N(T_3)_{v_3}\,.
\end{align*}
Thus, $N(G^{\prime };  T_0^{\prime })  >  N(G;  T_0 \cup   T_{k-1})$ provided that $n_0 >1$. This is indeed the case since $G \neq C_5$ and $T_0$ was chosen to have the maximum order among the trees $T_0,T_1,\ldots,T_{k-1}$.
\item If $k=4$, then
\begin{align*}
0 \leq N(G^{\prime };  T_0^{\prime })  - & N(G;  T_0 \cup   T_3)=(2^{n_3-1} -1 ) (2^{n_0-1} -1 )\\
&  + 2^{n_0 -1} (2^{n_3-1} - 1 ) \sum_{l=1}^2 \prod_{j=1}^l N(T_j)_{v_j}\\
&+ 2^{n_3 -1} (2^{n_0} - 1 -2^{n_1-1} ) N(T_2)_{v_2}\,.
\end{align*}
It is easy to see that $N(G;  T_0^{\prime })  -  N(G;  T_0 \cup   T_3)=0$ if and only if $n_0=n_1=n_3=1$.
\end{itemize}

Now, observe that the graphs $G-(V(T_0) \cup V(T_{k-1}))$ and $G^{\prime} -V(T_0^{\prime} ) $ are isomorphic, so they have the same number of connected induced subgraphs. Altogether, we conclude that $k= 3$, completing the proof of the claim.

\medskip
The specialisation $k=3$ equation~\eqref{OneG} yields
\begin{align*}
 N(G; T_0 \cup   T_2) & =N(T_0) + N(T_2) ~ + N(T_0)_{v_0} N(T_1)_{v_1} + N(T_2)_{v_2} N(T_1)_{v_1}\\
 & +  N(T_0)_{v_0} N(T_2)_{v_2} (1 + N(T_1)_{v_1})\,.
\end{align*}
Recall that $T_0$ has the maximum order among the trees $T_0,T_1,T_2$.

\textbf{Claim 3}: \textit{We have $n_1=n_2=1$}.

To see this, suppose (for contradiction) that $n_2>1$. Construct from $G$ a new unicyclic graph $G^{\prime}$ by replacing both $T_2$ with the star $T_2^{\prime}=S_{n_2-1}$ centered at vertex $v_2$, and $T_0$ with the star $T_0^{\prime}=S_{n_0+1}$ centered at vertex $v_0$. Thus, the number $N(G^{\prime}; T_0^{\prime} \cup   T_2^{\prime })$ of connected induced subgraphs of $G^{\prime}$ that contain a vertex of $T_0^{\prime}$ or $T_2^{\prime}$ is given by 
\begin{align*}
 N(G^{\prime}; T_0^{\prime} \cup   T_2^{\prime}) & =N(T_0^{\prime}) + N(T_2^{\prime}) ~ + N(T_0^{\prime})_{v_0} N(T_1)_{v_1} + N(T_2^{\prime})_{v_2} N(T_1)_{v_1}\\
 & +  N(T_0^{\prime})_{v_0} N(T_2^{\prime})_{v_2} (1 + N(T_1)_{v_1})\,,
\end{align*}
which implies that
\begin{align*}
N(G^{\prime}; T_0^{\prime} \cup   T_2^{\prime}) & - N(G; T_0 \cup   T_2) =N(T_0^{\prime}) + N(T_2^{\prime})  - N(T_0) - N(T_2) \\
& + ( N(T_0^{\prime})_{v_0} + N(T_2^{\prime})_{v_2} - N(T_0)_{v_0} - N(T_2)_{v_2} ) N(T_1)_{v_1}\\
& +  (N(T_0^{\prime})_{v_0} N(T_2^{\prime})_{v_2} - N(T_0)_{v_0} N(T_2)_{v_2}) (1 + N(T_1)_{v_1})\,.
\end{align*}
Again Proposition~\ref{upperTr} along with Theorem~\ref{Szek:Theo} gives
\begin{align*}
N(T_0^{\prime}) + N(T_2^{\prime})  - N(T_0) - N(T_2)=2^{n_0-1} - 2^{n_2-2}\,,\\
N(T_0^{\prime})_{v_0} + N(T_2^{\prime})_{v_2} - N(T_0)_{v_0} - N(T_2)_{v_2}=2^{n_0-1} - 2^{n_2-2}
\end{align*}
and
\begin{align*}
N(T_0^{\prime})_{v_0} N(T_2^{\prime})_{v_2} - N(T_0)_{v_0} N(T_2)_{v_2}=0\,.
\end{align*}
Therefore, we get
\begin{align*}
N(G^{\prime}; T_0^{\prime} \cup   T_2^{\prime}) & - N(G; T_0 \cup   T_2) =(2^{n_0-1} - 2^{n_2-2}) (1+ N(T_1)_{v_1}) > 0
\end{align*}
which completes the proof of the claim. This also completes the proof of the theorem.
\end{proof}

\section{Concluding comments}\label{conclude}
In this paper, we have studied the problem of finding a connected graph that minimises or maximises the number of connected induced subgraphs, as an extension of the number of subtrees of trees. The problem is considered for two different families of graphs, namely general connected graphs and (connected) unicyclic graphs. It is shown that the path uniquely realises not only the minimum number of connected induced subgraphs among all connected graphs of order $n$, but also the minimum number of $k$-vertex connected induced subgraphs. For all $k \in \{3,4,\ldots,n-1\}$, the maximum number of $k$-vertex connected induced subgraphs among all connected graphs of order $n$, however, is shown to be attained by all graphs obtained by removing some independent edges from the complete graph of order $n$. Considering the family of unicyclic graphs of order $n>5$, we have proved that the minimum number of connected induced subgraphs is uniquely attained by the tadpole graph $G_{3,n-3}$, while the graph resulting from adding one edge between two leaves of the star of order $n$ uniquely realises the maximum. Interestingly enough, the graphs with the minimum (resp. maximum) number of connected induced subgraphs were previously shown to maximise (resp. minimise) the Wiener index (sum of distances between all unordered pairs of vertices), in all families that we have investigated.

\medskip
Our results can also be generalised to graphs of order $n$ with $r$ connected components: in particular, the extremal graphs can be obtained. It is not difficult to see that the minimum number of connected induced subgraphs is uniquely realised by the disjoint union of the $r$ paths whose orders are all as equal as possible and sum to $n$, while the maximum is uniquely attained by the disjoint union of $r-1$ copies of the one-vertex graph and one copy of the complete graph of order $n-r+1$.

Now the following natural questions arise:
\begin{itemize}
\item What are the graphs having the minimum number of connected induced subgraphs among all connected series-reduced graphs (graphs without vertices of degree $2$) of order $n$?
\item Can we find a constructive characterisation of the graphs extremising the number of connected induced subgraphs over all connected graphs with prescribed order $n$ and number of cycles $d>1$? In general, this problem can appear out of reach due to the various ways in which the cycles may intersect in the graph. However, special cases of bicyclic graphs ($d=2$ or $3$) are still of interest. Note that the case $d=0$ corresponds to trees while the case $d=1$ has been treated in this paper.
\end{itemize}

\section*{Acknowledgements} 
The author expresses his appreciations to the referee for his/her careful reading of the manuscript and valuable comments, which have improved the presentation of this paper. The author is also indebted to the mathematics division of Stellenbosch University which was the starting point of this work.

\end{document}